\documentclass[10pt,a4paper]{article}
\usepackage{amsmath, amsfonts, amsthm, amssymb, amscd,enumerate}
\usepackage{graphicx, color}
\usepackage[mathcal]{euscript}
\usepackage{epstopdf}
\newtheorem{theorem}{Theorem}[section]

\newtheorem{lemma}{Lemma}[section]
\newtheorem{corollary}{Corollary}[section]
\newtheorem{nb}{Remark}[section]
\newtheorem{example}{Example}[section]

\numberwithin{equation}{section}

\newcommand{\ds}{\displaystyle}

\newcommand{\Ti}{T_{\rm i}}

\newcommand{\qq}{\left(\frac52+\frac{\Ti}{T}\right)}
\makeatletter

\newlength{\captionwidth}
\long\def\@makecaption#1#2{%
   \vskip 10\p@
   \setbox\@tempboxa\hbox{\small #1: #2}%
   \ifdim \wd\@tempboxa > \captionwidth 
       \hbox to\hsize{\hfil
       \parbox[t]{\captionwidth}{
        \small #1: #2\par}
       \hfil}
     \else
       \hbox to\hsize{\hfil\box\@tempboxa\hfil}%
   \fi}

\makeatother

\begin{document}
\title{The reflection of an ionized shock wave}
\author{Fumioki ASAKURA
\footnote{Center of Physics and Mathematics, Osaka Electro-Communication University, Neyagawa, Osaka, Japan,
 e-mail: {\tt\small asakura@osakac.ac.jp}}
\hspace{4ex}
Andrea CORLI
\footnote{Department of Mathematics and Computer Science, University of Ferrara, Via Machiavelli 30, 40121 Ferrara, Italy,
e-mail: {\tt\small andrea.corli@unife.it}}
}
\maketitle

\begin{abstract}
In a previous paper we studied the thermodynamic and kinetic theory for an ionized gas, in one space dimension; in this paper we provide an application of those results to the reflection of a shock wave in an electromagnetic shock tube. Under some reasonable limitations, which fully agree with experimental data, we prove that both the incident and the reflected shock waves satisfy the Lax entropy conditions; this result holds even outside genuinely nonlinear regions, which are present in the model.
We show that the temperature increases in a significant way behind the incident shock front but the degree of ionization does not undergo a similar growth. On the contrary, the degree of ionization increases substantially behind the reflected shock front. We explain these phenomena by means of the concavity of the Hugoniot loci. 
Therefore, our results not only fit perfectly but explain what was remarked in experiments.
\end{abstract}

\smallskip

\textit{\quad 2010~Mathematics Subject Classification:} 35L65, 35L67, 76N15.

\smallskip

\textit{\quad Key words and phrases:}
Systems of conservation laws, ionized gas, shock reflection.

\section{Introduction}
A shock tube is a long tube with constant cross-section having a high-pressure chamber and a low-pressure chamber separated by a diaphragm. In the high-pressure chamber, the gas (driver gas) is compressed and heated by moving forward a heavy piston. If the diaphragm bursts open, the driver gas induces a flow into the low-pressure chamber and produces a shock wave propagating in the same direction. The temperature and the pressure behind the shock wave can be measured by some instrument at the end of the tube if it is open. If the end is closed, the shock wave is reflected; also in this case, both the temperature and the pressure after the reflection can be measured as well.
\par
As long as mechanical devices are used for compressing the driver gas, the maximum shock speeds and temperatures are necessarily limited. In {\it electromagnetic shock tubes\/}, the driver gas is heated by means of an electric discharge and, sometimes, even accelerated by magnetic forces. The simplest electromagnetic shock tube, called the {\it T-tube\/}, is  shown schematically in Figure \ref{fig:T-tube} (without a reflector in most cases). The tube is filled with gas at low pressure. By switching on the circuit, the capacitor bank is discharged between the electrodes; the gas in the discharge region is rapidly heated to a high temperature, and hence let out by the large pressure into the glass tube at a high speed. A strong shock wave is thus produced and the gas behind the shock is ionized.  
\begin{figure}[hbt]
\centering
 \includegraphics[width =.65\linewidth]{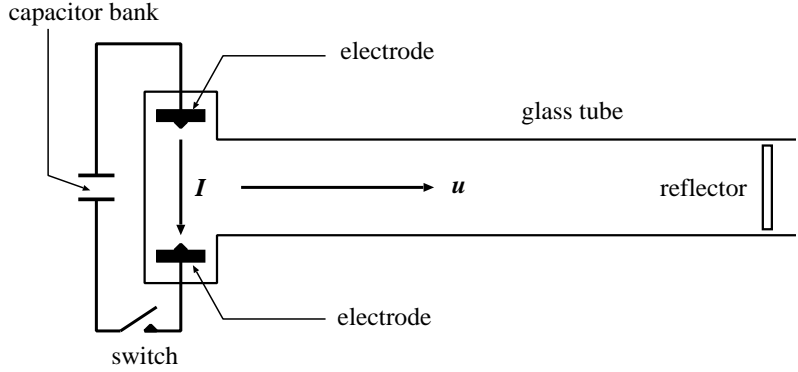}
 \caption{T-Tube with a reflector.}\label{fig:T-tube}
\end{figure}
\par
In 1960's, K. Fukuda and his colleagues made spectrometric measurements of the ionization generated by a shock wave in a plasma formed by helium or a mixture of helium-hydrogen gas. In particular, a reflector was set at the end of the T-tube and the measurements were done for the plasma generated by a shock reflection. For given initial data and speed of the shock wave, the temperature, particle density, and  degree of ionization of the plasma behind the shock wave were computed by the Rankine-Hugoniot conditions together with the Saha ionization formula, by assuming a condition of thermal equilibrium; this is called the {\it shock tube problem\/}.  As a basic reference of their analysis, they constructed several  ionization models in plasma and carried out computations in \cite{Fukuda-Okasaka-Fujimoto} \footnote{An English translation of \cite{Fukuda-Okasaka-Fujimoto} is available upon request to F. Asakura.}.
\par
The aim of this paper is to validate some results of \cite{Fukuda-Okasaka-Fujimoto} from a rigorous mathematical point of view by exploiting the theoretical analysis done in \cite{Asakura-Corli_RIMS, Asakura-Corli_ionized}. We briefly recall the physical model and the underlying assumptions; we refer to \cite{Asakura-Corli_ionized} for more details.

We  deal with a monatomic gas and assume that it can undergo only one ionization. We suppose that:

\begin{enumerate}[(1)] 

\item interaction potential energies are negligible with respect to kinetic energies;

\item effects of particle collisions can be neglected; 

\item local thermodynamic equilibrium is attained.

\end{enumerate}
We denote the concentration (number per unit volume) of atoms, ions and electrons in the gas by $n_{\rm a}$, $n_{\rm i}$ and $n_{\rm e}$, respectively. Let $m_{\rm p}$ denote the particle mass; by dropping the contribution of the electron mass we have $\rho = m_{\rm p}n_{\rm p}$, where $n_{\rm p} = n_{\rm a}+n_{\rm i}$ is the total number of atoms and ions per unit volume. The degree of ionization is defined by $\alpha = n_{\rm e}/(n_{\rm a} + n_{\rm i});$ we also denote the absolute temperature by $T$ and the Boltzmann constant by $k$. By  postulates (1), (2) and Dalton's law we deduce that the pressure $p$ assumes  the form 
\begin{equation}\label{eq:pressure-ion-mp}
p = (n_{\rm a}+n_{\rm i}+n_{\rm e})kT = (1+\alpha)n_{\rm p}kT = (1+\alpha)\frac{\rho}{m_{\rm p}}kT.
\end{equation}
Let $N_0$ denote Avogadro's number, $M = N_0 m_{\rm p}$ the molar mass  (denoted by $m$ in \cite{Asakura-Corli_ionized}), $R = k N_0$ the universal gas constant and $a^2=\frac{R}{M}$. Then, expression \eqref{eq:pressure-ion-mp} can we written as a modified equation of state of gas dynamics, namely, 
\begin{equation}\label{eq:EOS}
  p = a^2(1 + \alpha)\rho T.
\end{equation}
On the other hand, as a consequence of postulate (3), the equation of state is supplemented by Saha's equation
\begin{equation}\label{eq:saha}
\frac{n_{\rm i} n_{\rm e}}{n_{\rm a}} = \frac{2Z_{\rm i}}{Z_{\rm a}} \frac{(2\pi m_{\rm e} kT)^{\frac{3}{2}}}{h^3}\, e^{-\frac{T_{\rm i}}{T}}.
\end{equation}
As we shall see later, equation \eqref {eq:saha} can be put under the form given in \eqref{eq:deg-ionization-n}, by which $\alpha$ is determined by $p$ and $T.$
\par
We now lay down the plan of our paper. Section \ref{sec:ionized-gasdynamics} contains a brief account of the mathematical model together with the most important results proved in \cite{Asakura-Corli_ionized} that are used in the following. For our model, the shock tube problem consists in finding the thermodynamic state $(\alpha_-, T_-)$ behind the incident shock wave once the thermodynamic state $(\alpha_+, T_+)$ ahead of the shock wave and $u_{\pm}$, the particle velocities on both sides, are given. We proved in \cite{Asakura-Corli_ionized} that this problem has a unique solution. We also proved there that the forward and backward characteristic directions fail to be genuinely nonlinear in a small region; however, the thermodynamic part of the Hugoniot locus of a state always is the graph of a strictly increasing function of $\alpha$ in the $(\alpha, T)$-plane. 
\par
In an actual shock tube problem, the initial degree of ionization $\alpha_+$ is almost zero. However, Saha's equation \eqref{eq:deg-ionization-n} does not allow the value $\alpha_+ = 0$; this leads us to construct, in Section \ref{sec:incident shock}, an approximation of the thermodynamic part of the Rankine-Hugoniot condition. We show that, as long as this approximation is concerned, the Hugoniot curves stay in a genuinely nonlinear region, as is the case for both incident and reflected shock waves occurring in experiments. 
\par
In Section \ref{sec:temperature} we focus on the variation of the temperature across a shock wave. In particular, we identify two dimensionless parameters that are useful to estimate such a variation. This analysis is exploited in Section \ref{sec:condition Lax}, where we prove that the Lax shock conditions \cite{Smoller} are satisfied by both the incident and the reflected shock wave. In particular, we observe that in the state behind the {\em incident} shock front the temperature increases remarkably while the degree of ionization only a little; however, the degree of ionization increases much more in the state behind the {\em reflected} shock front, which shows the effective role played by the reflector set at the end of the T-tube. From a mathematical point of view, we explain this phenomenon by the concavity of the Hugoniot loci.

\section{System of ionized gasdynamics}\label{sec:ionized-gasdynamics}
\setcounter{equation}{0}
In this section we introduce the system of ionized gasdynamics and briefly recall the most important results of \cite{Asakura-Corli_ionized}. Under the notation already used in the Introduction, the system is
\begin{equation}\label{eq:system-n}
\left\{
\begin{array}{l}
\rho_t + (\rho u)_x = 0,
\\
\ds(\rho u)_t + (\rho u^2 + p)_x = 0,\rule{0ex}{2.75ex}
\\
\ds\left(\rho E\right)_t + \left(\rho u E + pu\right)_x = 0,\rule{0ex}{2.75ex}
\end{array}
\right.
\end{equation}
where $E = \frac{u^2}{2} + e$ is the (specific) total energy; $e$ is the (specific) internal energy and $u$ the particle velocity. By \eqref{eq:EOS} we deduce 
\begin{equation}\label{eq:e-n}
e  =  \frac32a^2(1+\alpha)T + a^2T_{\rm i}\alpha,
\qquad
H  =  \frac52a^2(1+\alpha)T + a^2T_{\rm i}\alpha, 
\end{equation}
where $H$ is the enthalpy. We also denote by $v$ and $S$ the specific volume and entropy, respectively, and by $\eta= S/a^2$ the dimensionless (specific) entropy; we use the notation $\theta = T(1 + \alpha)$.  In order to close system \eqref{eq:system-n}, in addition to the equation of state \eqref{eq:EOS} we need a further equation linking the ionization degree $\alpha$ to the pressure and the temperature. This is the famous Saha's equation, which can be written as 
\begin{equation}\label{eq:deg-ionization-n}
    \alpha = \left(1 + \kappa pT^{-\frac{5}{2}}e^{\frac{\Ti }{T}}\right)^{-\frac{1}{2}},
\end{equation}
where $T_{\rm i}>0$ is the ionization temperature and $\kappa>0$ a suitable constant. For the values of these and other constants we refer to Appendix \ref{sec:appendix_numerical_values}. We notice that $\alpha\in(0,1)$. Equation \eqref{eq:deg-ionization-n} may be inverted to deduce $p$ (and then $\rho$, $v$ or $\eta$) in terms of $\alpha$ and $T$. In particular one finds \cite[(3.9) and (3.14)]{Asakura-Corli_ionized}
\begin{equation}\label{eq:p-alphaT}
    p = p(\alpha, T) = \frac1\kappa\,\frac{1 - \alpha^2}{\alpha^2}T^{\frac{5}{2}}e^{-\frac{\Ti }{T}},
    \qquad
    \eta(\alpha,T) =  2\log \frac{\alpha}{1 - \alpha} +  (1 + \alpha)\qq  + C.
\end{equation}
\paragraph{Characteristic vector fields}
As in \cite{Asakura-Corli_ionized}, in the following we denote partial derivatives with subscripts; for instance, $\eta_\alpha$ and $\eta_T$ are the partial derivatives of $\eta=\eta(\alpha,T)$ with respect to $\alpha$ and $T$, respectively, and so on. System \eqref{eq:system-n} is strictly hyperbolic with eigenvalues $\lambda_1 = u - c,\  \lambda_0 = u,\  \lambda_2 =  u + c,$ and corresponding characteristic vectors $R_1 = (-1, c/\rho, 0)^T$, $R_0 = (0,  0,  1 )^T$, $R_2 = (1, c/\rho, 0)^T$. Here we denoted by $c=\lambda/\rho$ the sound speed, for $\lambda^2=-\eta_T/(v_p\eta_T-v_T\eta_p)$, see \cite[(4.9)]{Asakura-Corli_ionized}; we also have 
\begin{align}
  c &= \sqrt{p_\rho(\rho,S)} = 
a\sqrt{\frac{5\theta}{3}}\sqrt{\frac{1 + \alpha(1 - \alpha)\left(\frac{5}{4} + \frac{\Ti}{T} + \frac{\Ti^2}{5T^2}\right)}{1 + \alpha(1 - \alpha)\left(\frac{5}{4} + \frac{\Ti}{T} + \frac{\Ti^2}{3T^2}\right)}}
\label{eq:euler characteristic speed}
\end{align}
and
\begin{equation}\label{e:pS}
	p_\eta(\rho,\eta)  = \frac{2p\left[1 + \frac{1}{2}\alpha(1 - \alpha)\left(\frac{5}{2}  + \frac{T_i}{T}\right)\right]}{3(1 + \alpha)\left[ 1 + \alpha(1 - \alpha)\left(\frac{5}{4}+   \frac{T_i}{T} +  \frac{T_i^2}{3T^2} \right)\right]} > 0.
\end{equation}
By direct computation we find
\begin{equation}\label{eq:p_vv}
  R_1\cdot \nabla \lambda_1 = R_2\cdot \nabla \lambda_2 = \frac{\partial c}{\partial \rho} + \frac{c}{\rho} = \frac{p_{vv}}{2\rho^3\sqrt{-p_v}}.
\end{equation}
%
\par 
\paragraph{Genuinely nonlinear region}
While the $0$-characteristic direction is linearly degenerate, a notable feature of system \eqref{eq:system-n} is that the $1$ and $2$-characteristic directions miss to be genuinely nonlinear, see Figure \ref{f:GNL} on the left. However, we have the following result, see \cite[Th. 4.1]{Asakura-Corli_ionized}.
\begin{theorem}\label{thm:genuine nonlinearity}
If either $\alpha \leq 60 \left(\frac{T}{\Ti}\right)^{\! 3}$ or $\frac{\Ti}{T} \leq 54.5375,$
then the $1$ and $2$-characteristic directions are genuinely nonlinear.
\end{theorem} 
\begin{figure}[hbt]
\begin{tabular}{cc}
\includegraphics[width =.45\linewidth]{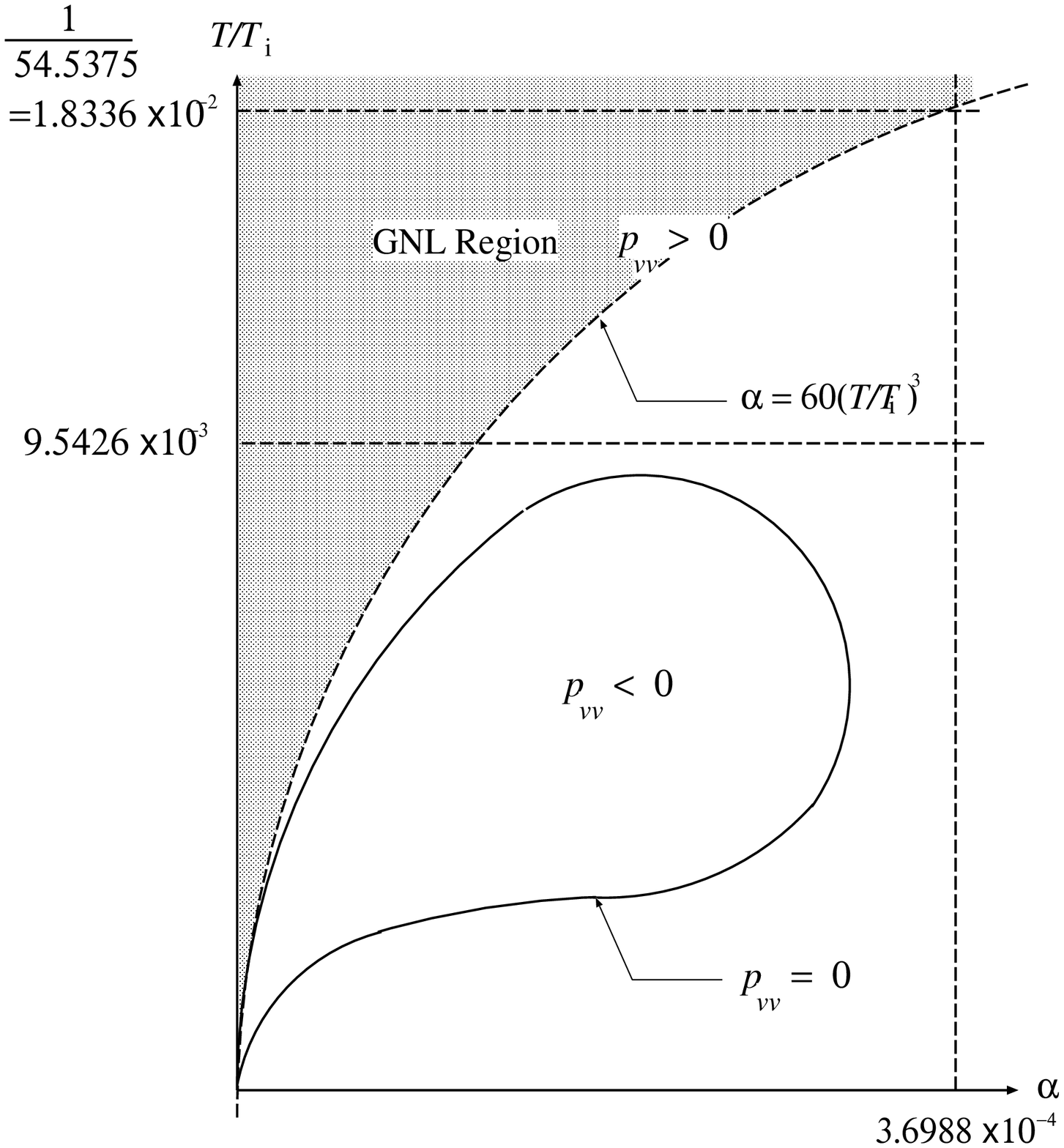}
&
\includegraphics[width =.55\linewidth]{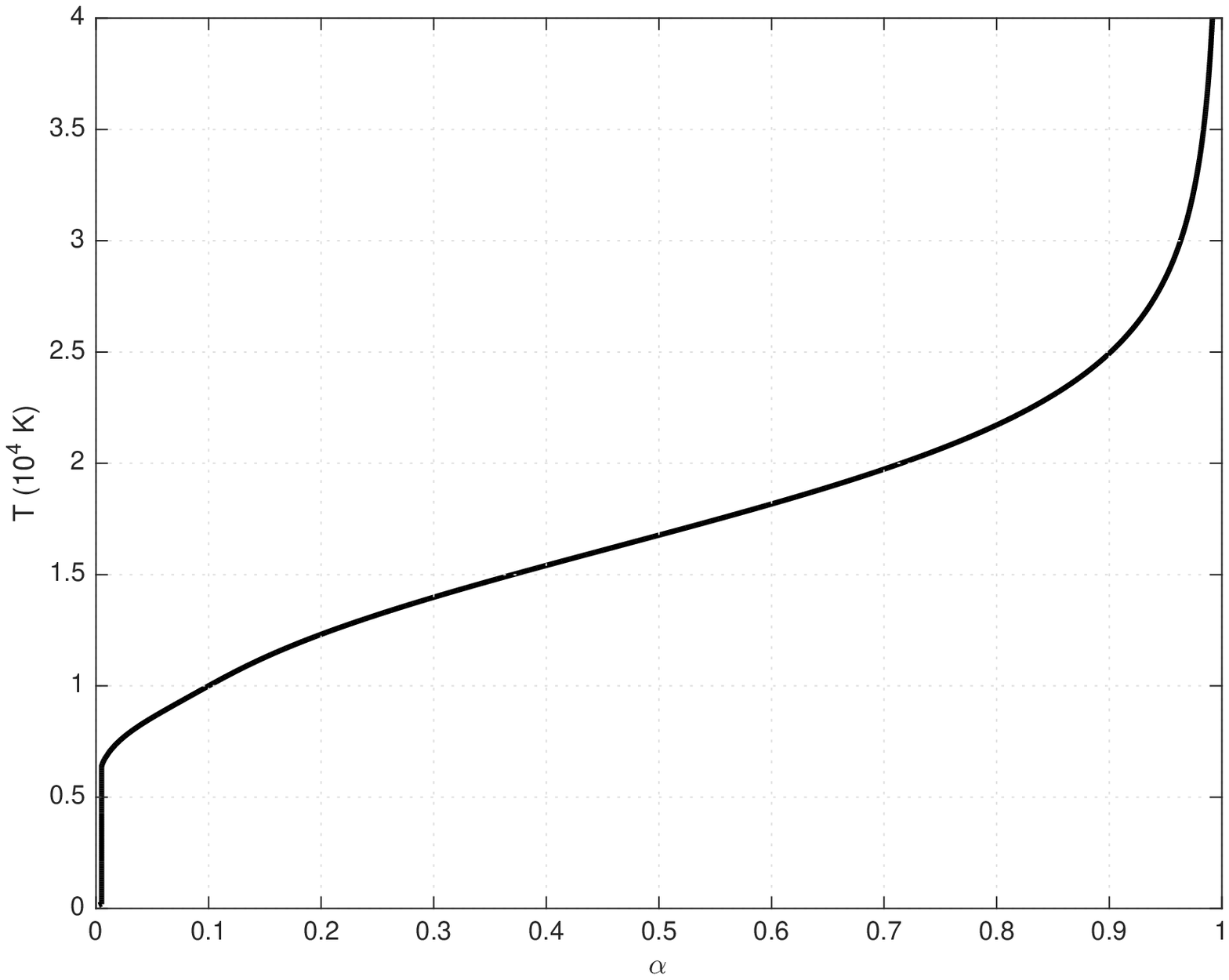}
\end{tabular}
\caption{Left: the $1$ and $2$-characteristic directions are not genuinely nonlinear on the curve in the white region (inflection locus); the region characterized by Theorem \ref{thm:genuine nonlinearity} is depicted in gray. Right: the thermodynamic part of the Hugoniot locus, see Theorem \ref{thm:Hugoniot curve alpha T}; data $\alpha_0$ and $T_0$ are as in case (1) of Appendix \ref{sec:appendix_numerical_values}.
}
\label{f:GNL}
\end{figure}
We call {\em genuinely nonlinear} a region where both $1$ and $2$-characteristic directions are genuinely nonlinear. We notice that the value $\alpha= 3.6988\times 10^{-4}$, see Figure \ref{f:GNL} on the left, is much larger than that occurring in the low-pressure chamber in the experiments, see Appendix \ref{sec:appendix_numerical_values}. We also notice that the condition $\frac{T}{T_{\rm i}}\ge 1/54.5375\sim 1.8336\times 10^{-2}$ is almost two times larger than the more precise value 
$\frac{T}{T_{\rm i}}\ge 9.5057\times 10^{-3}$ deduced numerically as an upper bound for the inflection locus.

\paragraph{Rankine-Hugoniot conditions}
The Rankine-Hugoniot conditions for a discontinuity of constant speed $U$ are
\begin{equation}\label{eq:RH-n}
\left\{
\begin{array}{l}
U[\rho] = [\rho u],
\\
U\ds[\rho u] = [\rho u^2 + p],
\\
U\ds[\rho E] = [\rho u E + pu].
\end{array}
\right.
\end{equation}
Here we denoted $[\rho]=\rho_+-\rho_-$, where $\rho_+$ and $\rho_-$ are the traces of $\rho$ at $x=Ut$ from the {\em right} and from the {\em left} side, respectively; the same notation is used for the other variables.
%
%
%
%

\noindent By computing $U$ from $\eqref{eq:RH-n}_1$, conditions $\eqref{eq:RH-n}_{2,3}$ can be written as
\begin{equation}\label{eq:Rankine-Hugoniot-n}
\left\{\begin{array}{rcl}
        (u_+ - u_-)^2  + (p_+ - p_-)(v_+ - v_-) & = & 0,\\[1mm]
        e_+ -  e_- + \frac{1}{2}(p_+ + p_-)(v_+ - v_-) & = & 0.
                \end{array}\right.
\end{equation}
For a fixed state $(\alpha_-,u_-,T_-)$, the states $(\alpha_+,u_+,T_+)$ satisfying system \eqref{eq:Rankine-Hugoniot-n} are said to form the {\em Hugoniot locus} of $(\alpha_-,u_-,T_-)$. Equation $\eqref{eq:Rankine-Hugoniot-n}_1$ and $\eqref{eq:Rankine-Hugoniot-n}_2$ are called the {\em kinetic part} and the {\em thermodynamic part} of the Hugoniot locus, respectively. In particular, equation $\eqref{eq:Rankine-Hugoniot-n}_2$ can be written in terms of $T$ and $\alpha$ as
\begin{align}
 &T_-\left[\left(4 + \frac{p_+}{p_-}\right)(1 + \alpha_-) +  \frac{2\Ti}{T_{-}}\alpha_-\right] 
=  T_+\left[\left(4 + \frac{p_-}{p_+}\right)(1 + \alpha_+) +  \frac{2\Ti}{T_{+}}\alpha_+\right],\label{eq:T_-/T_+prop-n}\\
&\frac{p_+}{p_-} 
= \left(\frac{1 - \alpha_+^2}{1 - \alpha_-^2}\right)\left(\frac{\alpha_-}{\alpha_+}\right)^2 \left(\frac{T_{+}}{T_{-}}\right)^{\frac{5}{2}}
e^{-\frac{\Ti}{T_{+}} + \frac{\Ti}{T_{-}}}. \label{eq:p_+/p_-}
\end{align}
Then, equations \eqref{eq:Rankine-Hugoniot-n} may be also written as 
\begin{align}
\label{eq:kinetic}
&\frac{p_-}{p_+} + \frac{v_-}{v_+} - 1 -  \frac{\theta_-}{\theta_+}   -  \frac{ (u_- - u_+)^2}{a^2\theta_+} = 0,
\\
\label{e:T-T+-n}
&    \left(4 + \frac{p_+}{p_-}\right)\theta_- + 2\Ti\alpha_- = \left(4 + \frac{p_-}{p_+}\right)\theta_+ + 2\Ti\alpha_+.
\end{align}
For future reference, when we fix states $(\alpha_+,u_+,T_+)$ and $u_-$, we denote by 
\begin{equation}\label{e:GH}
G_+(\alpha_-,T_-)=0\quad  \hbox{ and } \quad H_+(\alpha_-,T_-)=0 
\end{equation}
the loci in \eqref{eq:kinetic} and \eqref{e:T-T+-n}, respectively. The following result is a consequence of  \cite[Prop. 5.3, Th. 5.1]{Asakura-Corli_ionized}.
\begin{theorem}\label{thm:Hugoniot curve alpha T}
Fix $(\alpha_0,T_0)$. In the $(\alpha, T)$-plane, the thermodynamic part of the Hugoniot locus of $(\alpha_0,T_0)$ is the graph of a strictly increasing function $T=T(\alpha)$ for $\alpha\in (0,1)$, see Figure \ref{f:GNL} on the right, and $\lim_{\alpha\to0}T(\alpha)=0$, $\lim_{\alpha\to1}T(\alpha)=\infty$.
Moreover, fix  $u$ and $u_0$, with $u \ne u_0$. Then there exists a unique point $(\alpha, T)$ with $\alpha \in (\alpha_0,1)$, such that $(\alpha, u, T)$ belongs to the Hugoniot locus of $(\alpha_0,u_0,T_0)$.
 \end{theorem} 


The {\it relative (shock)  velocity\/} is defined by $V = U - u$, which is the shock velocity relative to the particle velocity; note that $u_+ - u_- = -(V_+ - V_-).$ By using $V$, conditions \eqref{eq:RH-n} can be written as
\begin{equation}\label{e:RH-enthalpy}
\left\{
\begin{array}{l}  
\rho_+ V_+ = \rho_-V_-, \\
p_+ + \rho_+ V_+^2= p_- + \rho_- V_-^2,\\
H_+ + \frac{1}{2}V_+^2 = H_- + \frac{1}{2}V_-^2. 
\end{array}
\right.
\end{equation} 
Note that the {\it mass flux\/} $m = \rho_{\pm} V_{\pm}$ is constant. Since $\rho_{\pm}V_{\pm}^2 = m^2/\rho_{\pm} = m^2 v_{\pm},$ by $\eqref{e:RH-enthalpy}_2$ we have $m^2 = -(p_{+} - p_{-})/(v_{+} - v_{-})$, showing that $m$ is the {\it Lagrangian shock velocity\/}. Moreover
\(U = u_{\pm} + m/\rho_{\pm}.\)
Then, equations \eqref{e:RH-enthalpy} become  
\begin{equation}\label{eq:Rankine-Hugoniot V}
\left\{\begin{array}{l}
               (V_+ - V_-)^2 =- (p_+ - p_-)(v_+ - v_-),
               \\
               H_+ -  H_- = \frac{1}{2}(p_+ - p_-)(v_+ + v_-).
        \end{array}\right.
\end{equation}
\par
\paragraph{Lax conditions}
We call backward (forward) shock wave a shock wave corresponding to the 1- (resp., 2-) characteristic direction. The Lax conditions for backward and forward shock waves are, respectively,
\begin{align}
\label{eq:shock_conditions1}
\hbox{backward:\quad } &u_+ - c_+ < U  < u_+, \quad U < u_- - c_-,
\\
\label{eq:shock_conditions}
\hbox{forward:\quad } &u_- < U < u_- + c_-, \quad u_+ + c_+ < U.
\end{align}
Note that conditions \eqref{eq:shock_conditions} can be written as
\begin{equation}\label{eq:shock_conditions_2}
0 < V_-  <  c_-, \quad  c_+ < V_+.
\end{equation} 

\begin{nb}\label{nb:positive backward}
Conditions \eqref{eq:shock_conditions} and \eqref{eq:shock_conditions1} are called {\em evolutionary conditions} in Landau-Lifshitz \cite[\S 88]{Landau-Lifshitz}. Notice that, in Eulerian coordinates, a backward shock wave may have a positive speed $(U > 0).$ In this case, by \eqref{eq:shock_conditions1} we have $0 < U < u_+.$  Analogously, if a forward shock wave has a negative speed, then $u_- < U < 0$ by \eqref{eq:shock_conditions}. In Section \ref{sec:condition Lax} we assume $u_+ = 0$ for the incident shock wave and $u_- > 0$ for the reflected wave, see Figure \ref{fig:wall2}. Hence, these backward (forward) shock waves with positive (resp., negative) speed are ruled out in our analysis. 
\end{nb}

The following theorem is called the Bethe-Weyl theorem; we refer to  Menikoff-Plohr \cite[Th. 4.1]{Menikoff-Plohr}.
\begin{theorem}\label{thm:Bethe-Weyl}
The thermodynamic part of the Hugoniot locus of the state $(v_0,S_0)$ in $(p,v)$-plane intersects each isentrope at least once. If $p_{vv} > 0$ along an isentrope, then the thermodynamic part of the Hugoniot locus intersects it exactly once; in this case, we have $v < v_0$ and $|U - u| < c$ if $S > S_0,$ while the opposite inequalities hold if $S < S_0.$  
\end{theorem}
\par
Since \eqref{eq:euler characteristic speed} and \eqref{e:pS} hold,  Smoller \cite[Th. 18.3]{Smoller} claims the converse of Theorem \ref{thm:Bethe-Weyl} holds: namely, if $p_{vv} > 0$, then $S$ is monotone along the thermodynamic part of the Hugoniot locus.

\begin{nb}\label{nb:invariance_1}
Condition $p_{vv} > 0$ is equivalent to $R_1\cdot \nabla \lambda_1 = R_2\cdot \nabla \lambda_2 > 0$ by \eqref{eq:p_vv}. 
\end{nb} 

\paragraph{Integral curves in the $(\alpha,T)$-plane}
The concavity of the projections of the integral curves in the $(\alpha,T)$-plane has been partly established in \cite{Asakura-Corli_ionized}, where by concavity we mean the concavity as functions of the variable $\alpha$. Here we provide a more precise result.

\begin{theorem}\label{thm:concave T}
The projection of any integral curve on the $(\alpha,T)$-plane is strictly convex for $\frac{T_{\rm i}}{T} \leq 4$ and strictly concave in the region $\frac{T_{\rm i}}{T} >  37.5964$, $0 < \alpha \leq0.25$.
\end{theorem}
\begin{proof}
In \cite[Lemma 7.1]{Asakura-Corli_ionized} we already proved that the projection of any integral curve on the $(\alpha,T)$-plane is strictly convex for $\frac{T_{\rm i}}{T} \leq 4$ and strictly concave for small $\alpha.$
Then, we only have to prove the second part of the statement. In \cite{Asakura-Corli_ionized} we denoted by $T=\mathcal{T}(\alpha)$ the integral curve through $(\alpha_0,T_0)$ and found 
\begin{equation}\label{eq:T''}
   \frac{d^2\mathcal{T}(\alpha)}{d\alpha^2} =
-\left. \frac{\eta_{\alpha\alpha}\eta_T^2 - 2 \eta_{\alpha T}\eta_{\alpha}\eta_T + \eta_{TT}\eta_{\alpha}^2}{\eta_T^3}\right|{}_{\left(\alpha,\mathcal{T}(\alpha)\right)},
\end{equation}
where
\begin{align*}
\eta_{\alpha\alpha}\eta_T^2 - 2 \eta_{\alpha T}\eta_{\alpha}\eta_T + \eta_{TT}\eta_{\alpha}^2
&= \frac{2(1+\alpha)}{\alpha^2(1-\alpha)^2}\frac{T_{\rm i}}{T}\left\{-\left[1 - 3\alpha - \frac{5}{2}\alpha^2(1 - \alpha)^2\right]\frac{T_{\rm i}}{T} + 4\left[1 + \frac{5}{4}\alpha(1 - \alpha)\right]^{\! 2}\right\},
\end{align*}
and $\eta_T= - \frac{T_{\rm i}}{T^{2}}(1 + \alpha)$. It is easy to check that $P(\alpha) = 1-3\alpha -\frac52\alpha^2(1-\alpha)^2$ is strictly decreasing in the interval $[0,1]$ and $P(0)=0$, $P(1)=-2$. Then, there is a single point $\bar \alpha\sim 0.2970$ where $P(\bar\alpha)=0$. If $\alpha<\bar\alpha$, then the right side of \eqref{eq:T''} is negative if 
\begin{equation}\label{eq:T''<0}
\frac{T_{\rm i}}{T} >  \frac{\left[2 + \frac{5}{2}\alpha(1 - \alpha)\right]^{\! 2}}{1 - 3\alpha - \frac{5}{2}\alpha^2(1 - \alpha)^2}.
\end{equation}
For $\alpha\in[0,\bar\alpha)$ the denominator $P(\alpha)$ decreases while the numerator increases;  then the right-hand side of \eqref{eq:T''<0} increases. Therefore, fix $\alpha_1 = 1/4=0.25 < 0.2970$ to make things simple; at this point the right-hand side of \eqref{eq:T''<0} equals $37.5964$. The theorem is proved.
\end{proof}
The following result follows by the well known fact that the integral curve through $(\alpha_0,T_0)$ and the Hugoniot curve issuing from the same point have a second-order contact at  $(\alpha_0, T_0)$. 

\begin{corollary}\label{cor:concave T} If $(\alpha_0, T_0)$ satisfies the constraints in Theorem \ref{thm:concave T}, then the thermodynamic part of the Hugoniot curve  of $(\alpha_0, T_0)$ is strictly concave in a neighborhood of $(\alpha_0, T_0).$
\end{corollary}

\section{Approximate Hugoniot loci and genuine nonlinearity}\label{sec:incident shock}
%
%
As we showed in the previous section, the $1$ and $2$-characteristic directions are not genuinely nonlinear in a small zone close to the origin in the $(\alpha,T)$-plane. Indeed, this region is avoided in the experiments in \cite{Fukuda-Okasaka-Fujimoto} and our theoretical construction in the next section makes precisely this assumption. In order to justify this assumption, in this section we provide, by an approximation argument, some explicit conditions for the incident shock to be in a genuinely nonlinear region.

More precisely, we deal with  Regime 2 in \cite{Fukuda-Okasaka-Fujimoto}. With reference to Figure \ref{fig:shock reflection},  the velocity of the particles on the back side of the incident shock wave, coincides with the velocity of the head of a beam (or arc) in the electromagnetic shock tube, which corresponds to the velocity of the head of the piston for a classical shock tube. We assume that the right state ($+$ subscript) is at chamber temperature $T_+>0$ and pressure $p_+>0$; moreover, the gas is at rest, i.e., $u_+=0$. There, the degree of ionization is almost $0$ but we cannot assume $\alpha_+=0$:  this is precluded by Saha's law \eqref{eq:deg-ionization-n}.  This issue can also be noticed in the thermodynamic part of the Rankine-Hugoniot condition \eqref{eq:p_+/p_-}, where the pressure becomes singular for $\alpha_+=0$.  However, for $T_{+}$ in the range of the experiments of \cite{Fukuda-Okasaka-Fujimoto}, the ionization degree $\alpha_+$ is extremely small, and this is also the case of the term $\exp(-\Ti/(2T_{+}))$,
see Appendix \ref{sec:appendix_numerical_values}, which is comparable to $\alpha_+$. Then, in order to simplify the problem without losing any important information, we now show a way of approximating the Hugoniot locus that exploits this remark.

\begin{figure}[hbt]
\centering
 \includegraphics[width =.60\linewidth]{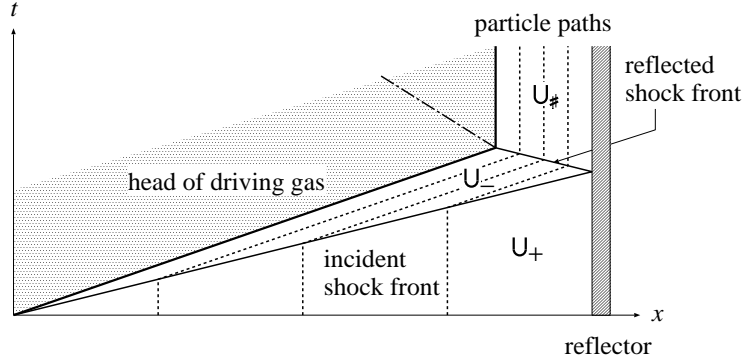}
 \caption{Reflection of a shock wave.}\label{fig:shock reflection}
\end{figure}
\paragraph{Approximate Hugoniot loci}
Fix $T_+,p_+$; Saha's formula \eqref{eq:deg-ionization-n} for the $+$ states can be written as  
\begin{equation}\label{e:alpha_+A}
   \alpha_+ = A_+e^{-\frac{\Ti}{2T_{+}}}\quad \hbox{ for }\quad  A_+ = A(\alpha_+,T_+)=\left(\kappa p_{+} T_{+}^{-\frac{5}{2}} + e^{-\frac{\Ti}{2T_{+}}}\right)^{-\frac{1}{2}}.
\end{equation}
By  \eqref{eq:p-alphaT} we have $p_+= \frac1\kappa\,(1 - \alpha_+^2)A_+^2T_+^{\frac{5}{2}}$ and identity \eqref{eq:p_+/p_-} becomes
\begin{equation}\label{e:pA}
\frac{p_+}{p_-} 
= \frac{\alpha_-^2}{A_+^2}\left(\frac{1 - \alpha_+^2}{1 - \alpha_-^2}\right) \left(\frac{T_{+}}{T_{-}}\right)^{\frac{5}{2}} e^{\frac{\Ti}{T_{-}}}.
\end{equation}
The data in Appendix \ref{sec:appendix_numerical_values} show that both $\alpha_+$ and $\exp(-\Ti/T_{+})$ are extremely small but comparable; in those cases we compute $A_+ = 5.9556$ and $A_+ = 18.7224$, respectively. Therefore, we assume that
\begin{equation}\label{eq:approximation alpha_+ = 0}
  \alpha_+ \sim 0, \quad e^{-\frac{\Ti}{T_{+}}} \sim 0.
\end{equation}
\par\noindent
In the above approximation we have, by \eqref{eq:p-alphaT},
\begin{equation}\label{eq:v_+ e_+}
     v_+ = \frac{RT_+}{Mp_+}, \quad e_+ = \frac{3RT_+}{2M}, \quad
   S_+ = \frac{R}{M}\left(\log p_+ -\frac{5}{2}\log T_+ \right) + \text{Const.}
\end{equation}
All quantities involved in \eqref{eq:p-alphaT} are well defined in the approximation \eqref{eq:approximation alpha_+ = 0}.
By \eqref{eq:p-alphaT} and \eqref{e:alpha_+A} we deduce
\begin{equation}\label{e:hatA+}
A_+ \sim \widehat{A}_+ = \sqrt{\frac{T_{+}^{5/2}}{\kappa  p_{+}}}.
\end{equation}
The values of $\widehat{A}_+$ differ from those of $A_+$ for the last digit only. Conditions \eqref{eq:T_-/T_+prop-n}-\eqref{eq:p_+/p_-} become
\begin{align}\label{eq:RH alpha_+ = 0}
&  \frac{T_-}{T_+}\left[\left(4 + \frac{p_+}{p_-}\right)(1 + \alpha_-) +  \frac{2\Ti}{T_{-}}\alpha_-\right] 
=  4 + \frac{p_-}{p_+},
\\
&\frac{p_+}{p_-} 
= \frac{\alpha_-^2}{\widehat{A}_+^2(1 - \alpha_-^2)} \left(\frac{T_{+}}{T_{-}}\right)^{\frac{5}{2}} e^{\frac{\Ti}{T_{-}}}. 
\label{eq:approximate p_+/p_-}
\end{align}
Equation \eqref{eq:RH alpha_+ = 0}, with $p_+/p_-$ provided by \eqref{eq:approximate p_+/p_-}, is considered as an {\em approximation} of the thermodynamic part of the Rankine-Hugoniot condition; notice that the term $p_+/p_-$ is no more singular for $\alpha_+=0$. We emphasize that assumption \eqref{eq:approximation alpha_+ = 0} is just an approximation of the thermodynamic part of the Rankine-Hugoniot conditions at low degree of ionization and temperature; it is not an approximation of Saha's equation, as we did in \cite{Asakura-Corli_ionized} for the High-Temperature-Limit model.

We recall that the pressure is a strictly increasing function of $\alpha$ along the Hugoniot curve \cite[Prop. 5.4]{Asakura-Corli_ionized} and then we may assume that also in the approximation \eqref{eq:approximation alpha_+ = 0} the limit $p_-(\alpha,T(\alpha))\to p_*$ exists for $\alpha_-\to0$, for some $p_*$. Passing to the limit for $\alpha_-\to0$ and $T_-\to T_+$ in \eqref{eq:RH alpha_+ = 0} we find $p_*=p_-$. At last, to obtain an asymptotic form of the Hugoniot locus as $\alpha_- \to 0$ and $T_{-} \to T_{+}>0,$ we have to assume that condition \eqref{eq:approximation alpha_+ = 0} also holds for $\alpha_-, T_{-}$. Thus, by passing to the limit for $\alpha_-\to0$ in \eqref{eq:approximate p_+/p_-} we deduce $\alpha_- \sim \widehat{A}_+\exp(-\Ti/(2T_{-})).$
\begin{theorem}\label{thm:GN region}
Fix $T_+$, $p_+$ and assume the approximation condition \eqref{eq:approximation alpha_+ = 0}. If $T_->T_+$, then
\begin{equation}\label{e:alphaless}
0 < \alpha_-  
< \widehat{A}_+ \left(\frac{T_{-}}{T_{+}}\right)^{\frac{3}{4}}e^{-\frac{T_{\rm i}}{2T_{-}}} 
= \sqrt{\frac{T_{+}T_{\rm i}^{\frac{3}{2}}}{\kappa p_{+} }}\left(\frac{T_{-}}{T_{\rm i}}\right)^{\frac{3}{4}}e^{-\frac{T_{\rm i}}{2T_{-}}}.
\end{equation}
Moreover, if the condition 
\begin{equation}\label{e:asyT}
  \sqrt{\frac{T_{+}T_{\rm i}^{\frac{3}{2}}}{\kappa p_{+} }} \leq 5.1670\times 10^{9}
\end{equation}
is fulfilled, then the Hugoniot curve issuing from $(p_+, T_+)$ is located in a genuinely nonlinear region.
\end{theorem}
\begin{proof}
We denote 
$$
\chi = \frac{\alpha_-^2}{\widehat{A}_+^2(1 - \alpha_-^2)} \left(\frac{T_{-}}{T_{+}}\right)^{-\frac{5}{2}}
e^{\frac{T_{\rm i}}{T_{-}}} = \frac{\kappa p_+\alpha_-^2}{(1 - \alpha_-^2)T_{-}^{\frac{5}{2}}}
e^{\frac{T_{\rm i}}{T_{-}}}.
$$
Then \eqref{eq:RH alpha_+ = 0} turns out to be
$$
\frac{T_{-}}{T_{+}}\left[4(1 + \alpha_-) + (1 + \alpha_-)\chi + \frac{2T_{\rm i}}{T_{-}}\alpha_-\right]
=
4 + \frac{1}{\chi},
$$
which gives the following quadratic equation for $\chi$
\begin{equation}\label{e:Gamma}
  \Gamma(\chi) = (1 + \alpha_-)\frac{T_-}{T_+}\chi^2 + 2 \left[2(1 + \alpha_-)\frac{T_-}{T_+} + \frac{T_{\rm i}}{T_+}\alpha_- - 2\right]\chi - 1 = 0.
\end{equation}
This equation has a unique positive root. Since $\frac{T_-}{T_+} \geq 1$, then $\Gamma(0) = -1 < 0$ and 
\begin{align*}
   \Gamma\left(\frac{T_+}{T_-}\right)
&= \alpha_-\frac{T_+}{T_-} + 2 \alpha_-\left(\frac{2T_-}{T_+} + \frac{T_{\rm i}}{T_+}\right)\frac{T_+}{T_-} + 3\left(1 - \frac{T_+}{T_-}\right)> 0.
\end{align*}
We conclude that $0 < \chi \leq \frac{T_+}{T_-}$, which yields \eqref{e:alphaless}. To prove the second part of the statement, we denote
\[
B = \sqrt{\frac{T_{+}T_{\rm i}^{3/2}}{\kappa p_{+} }}, \qquad g(x) = x^\frac34 e^{-\frac{1}{2x}}.
\]
By Theorem \ref{thm:genuine nonlinearity}, if $\frac{T}{T_{\rm i}} > \frac{1}{54.5375}$, then the statement is true. Since \eqref{e:alphaless} can be written as $0 < \alpha_- < B\, g(T_-/T_{\rm i})$,  we are left to prove 
\begin{equation}\label{e:superclaim}
B\, g\left(\frac{T_-}{T_{\rm i}}\right) \le 60 \left(\frac{T_-}{T_{\rm i}}\right)^3\quad  \hbox{ for } \quad\frac{T_-}{T_{\rm i}}\le \frac{1}{54.5375}.
\end{equation}
Condition \eqref{e:asyT} is equivalent to require that the inequality in \eqref{e:superclaim} is satisfied at $\frac{T_-}{T_{\rm i}} = \frac{1}{54.5375}$.
To prove \eqref{e:superclaim} for $\frac{T_-}{T_{\rm i}} < \frac{1}{54.5375}$, by setting $\xi = \frac{1}{x}$ and $f(\xi) = \xi^{\frac{9}{4}}e^{-\frac{\xi}{2}}$, we equivalently need to prove that 
\begin{equation}\label{e:fff}
   f(\xi) < 60/B\quad \hbox{ for }\quad \xi > 54.5375. 
\end{equation}
The function $f$ has a unique maximum at $\xi = \frac{9}{2}$ and decreases if $\xi > \frac{9}{2}$.
By \eqref{e:asyT}, we have $f(54.5375)= 1.1612 \times 10^{-8}\le 60/B$.
This proves \eqref{e:fff} and then the theorem. 
\end{proof}
Condition \eqref{e:asyT} is largely satisfied by the experimental data, see Appendix \ref{sec:appendix_numerical_values}. Therefore, in the following we can always think that genuine nonlinearity holds in the regions we consider. We point out, however, that most of the computations below do not rely on this assumption.

\section{Variation of the temperature across a shock wave}\label{sec:temperature}
In this section we study the variation of the temperature caused by an incident shock wave. We are mainly concerned with a forward (incident) shock wave but, by interchanging the role of the left and right states, all results are true for a backward (reflected) shock wave. Let us fix $(\alpha_+,u_+,T_+)$;
for a state $(\alpha_-,u_-, T_-)$ we denote
\begin{equation}\label{e:ThetaDeltaD}
   \Theta = \frac{\theta_-}{\theta_+} - 1, \quad
   d = \frac{\Ti}{\theta_+}(\alpha_- - \alpha_+), \quad
   D =  \frac{ (u_- - u_+)^2}{a^2\theta_+}.
\end{equation}
\begin{lemma}\label{l:Theta>0}
Let $(\alpha_-,u_-,T_-)$ and $(\alpha_+,u_+,T_+)$ be connected by a forward shock wave. Then
\begin{equation}\label{eq:condition Delta}
  d < \frac{D}{2}\sqrt{1 + \frac{4}{D}}
\end{equation}
and 
\begin{equation}\label{eq:Theta2}
     \Theta = \frac{D - 8d + 2\sqrt{(2D - d)^2 + 15D}}{15}>0.
\end{equation}
\end{lemma}
\begin{proof}

If we denote $\Pi = \frac{p_{-}}{p_{+}}$ then, by \eqref{e:T-T+-n}, $\Pi$ solves the quadratic equation \cite[(5.15)]{Asakura-Corli_ionized}
$$
   \theta_+\Pi^2 - 2\left[2(\theta_- - \theta_+) + \Ti(\alpha_- - \alpha_+)\right]\Pi - \theta_- = 0
$$
and then \cite[(5.16)]{Asakura-Corli_ionized}
\begin{equation}
  \Pi = 2\left(\frac{\theta_-}{\theta_+} - 1\right) + \frac{\Ti}{\theta_+}(\alpha_- - \alpha_+) + \sqrt{\left[2\left(\frac{\theta_-}{\theta_+} - 1\right) + \frac{\Ti}{\theta_+}(\alpha_- - \alpha_+)\right]^{\!2} + \frac{\theta_-}{\theta_+}}. \label{eq:p_-/p_+ RH}
\end{equation}
On the other hand, by \eqref{eq:kinetic} we have
$$
\frac{p_{-}}{p_{+}} + \frac{v_{-}}{v_{+}} = \frac{p_{-}}{p_{+}} + \frac{p_{+}\theta_-}{p_{-}\theta_+}= 1 +  \frac{\theta_-}{\theta_+} +   \frac{ (u_- - u_+)^2}{a^2\theta_+}. 
$$
Therefore $\Pi$ also satisfies the quadratic equation
$$
 \Pi^2 - \left[1 +  \frac{\theta_-}{\theta_+} +   \frac{ (u_- - u_+)^2}{a^2\theta_+} \right]\Pi + \frac{\theta_-}{\theta_+} = 0
$$
and then another expression of $\Pi$ is 
\begin{equation}\label{e:Pisecond}
    \Pi = \frac{1}{2}\left[1 +  \frac{\theta_-}{\theta_+} +   \frac{ (u_- - u_+)^2}{a^2\theta_+} \right] \pm \sqrt{\frac{1}{4}\left[1 +  \frac{\theta_-}{\theta_+} +   \frac{ (u_- - u_+)^2}{a^2\theta_+} \right]^{\! 2} -  \frac{\theta_-}{\theta_+}}.
\end{equation}
One of these two roots must coincide with \eqref{eq:p_-/p_+ RH}.
We notice that \cite[Prop. 5.2 and Th. 5.1]{Asakura-Corli_ionized} imply $T_{-}>T_{+}$ and $\alpha_->\alpha_+$ for given $(\alpha_+,T_{+})$; hence $\theta_->\theta_+$ and then $\Theta>0$. If we denote
$$
    A = 2\left(\frac{\theta_-}{\theta_+} - 1\right) + \frac{\Ti}{\theta_+}(\alpha_- - \alpha_+), \quad
    B = 1 +  \frac{\theta_-}{\theta_+} +   \frac{ (u_- - u_+)^2}{a^2\theta_+}, \quad
    C = \frac{\theta_-}{\theta_+},
$$ 
then, by equating \eqref{eq:p_-/p_+ RH} with \eqref{e:Pisecond} we get
$
   2A + 2\sqrt{A^2 + C} = B \pm\sqrt{B^2 - 4C}.
$
By squaring two times and noticing that $A = 2\Theta + d,\, B = 2 + \Theta + D$, we find a quadratic equation for $\Theta:$
\begin{equation}\label{eq:eq_Theta}
  15 \Theta^2 - 2(D - 8d) \Theta + 4d^2 - D^2 - 4D = 0.
\end{equation}
Equation \eqref{eq:eq_Theta} has two distinct real roots, one of which must be positive, since we already proved $\Theta>0$. If both roots were positive, we should have $D/8 > d$ and $4d^2 - D^2 - 4D > 0$
and hence
$$
   0 < \frac{D^2}{16} - D^2 - 4D = - \frac{15}{16} D^2 - 4D, 
$$
which is a contradiction. Thus one root is negative and $4d^2 - D^2 - 4D < 0.$ This condition is equivalent to \eqref{eq:condition Delta} and the positive root of \eqref{eq:eq_Theta} is given by \eqref{eq:Theta2}.
\end{proof}

\begin{corollary}\label{c:ThetaD}
If $D\ge\frac13$, then $\Theta<D$.
\end{corollary}
\begin{proof}
First, we claim that 
\begin{equation}\label{e:Theta-claim}
	\Theta < \frac{D + 2\sqrt{4D^2 + 15D}}{15} .
\end{equation}
Indeed, if $0 < d \leq 2D,$ then \eqref{e:Theta-claim} follows easily. If $d > 2D,$ then the claim follows by
\begin{align*}
	\Theta &\leq \frac{D - 8d + 2(d - 2D) + 2\sqrt{15D}}{15}  <  \frac{2\sqrt{15D}}{15} < \frac{D + 2\sqrt{4D^2 + 15D}}{15} .
\end{align*}
Since $D\ge\frac13$, then $15D = 45\times \frac{1}{3} D \leq 45 D^2$
and by \eqref{e:Theta-claim} it follows $\Theta < (D + 2\sqrt{4D^2 + 45D^2})/15 = D$.
\end{proof}

\begin{example}\label{ex:hydrogen_gas} We choose $T_{+} = 300 ${\rm K}, $u_+=0$ and $u_- = 1.6 \times 10^4${\rm ms}${}^{-1}$. We compute
$$
  \frac{\Ti}{T_{+}} =  5.260 \times 10^2, \quad
  D = \frac{u_-^2}{a^2T_{+}} = 1.026 \times 10^2. 
$$
Thus, condition \eqref{eq:condition Delta} is
$
      d = 5.260 \times 10^2 \times \alpha_- < 5.2309 \times 10, 
$
showing $0 < \alpha_- < 9.7434 \times 10^{-2}.$ 
By \eqref{eq:Theta2} we deduce, since $d\ge0$,
\begin{equation}\label{e:thetaless}
1 < \frac{\theta_-}{\theta_+} < 1+ \frac{D+2\sqrt{4D^2+15D}}{15} = 3.5708 \times 10.
\end{equation}
Thus 
$
T_- < \theta_- < 9.7434\times 10^2
${\rm K}.
Under the approximation \eqref{eq:approximation alpha_+ = 0}, by Theorem \ref{thm:GN region} we deduce
$\alpha_- < 
2.4650\times 10^{-2}. 
$
Then $d < 1.2966 \times 10$  
and \eqref{eq:Theta2} gives precisely $\theta_-/\theta_+ > 2.7096\times 10$. Thus $\theta_- > 8.1290\times10^3$ and then $T_->
7.9334{\rm K}.
$ 
The exact result can be computed numerically and are $\alpha_- = 0.0109$, $T_{-} = 9559.53\,{\rm K}$, see Figure \ref{fig:Hugoniot+G(a,T)_16000}.  Notice that $\frac{\Ti}{T_-} = 16.5071$ and then
$
   0 < \alpha_- < 1.338 \times 10^{-2}=60(T_-/T_{\rm i})^3.
$
Thus the part of the Hugoniot curve in discussion is located in a genuinely nonlinear region.

\begin{figure}[hbt]
\centering
\includegraphics[width =.60\linewidth]{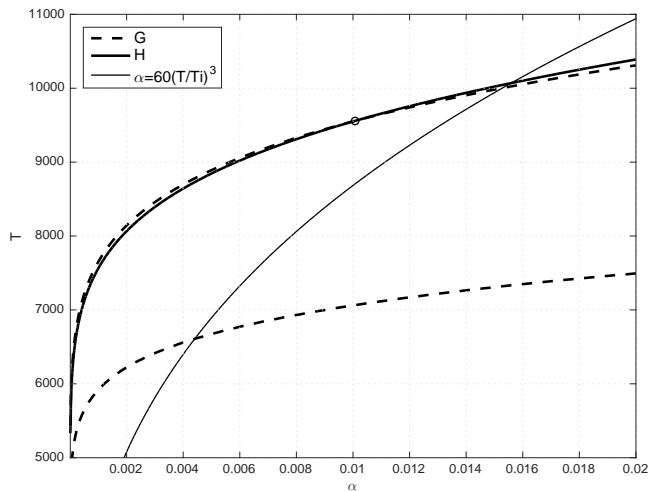}
\caption{The incident shock wave: the Hugoniot locus $H_+(\alpha,T)=0$ and $G_+(\alpha, T) = 0$, see \eqref{e:GH}. Here $(\alpha_+,u_+,T_+)$ are as in case (1), Appendix \ref{sec:appendix_numerical_values}, and $u_- = 1.6\times10^4\,{\rm m s}{}^{-1}$. The intersection point between the two loci, here represented with a small circle, is $\alpha_- = 0.0109$, $T_{-} = 9559.53\,{\rm K}.$ Notice that it lies where genuine nonlinearity holds, see Theorem \ref{thm:genuine nonlinearity}.}\label{fig:Hugoniot+G(a,T)_16000}
\end{figure}

\end{example}

\par
The following result directly follows by \eqref{eq:Theta2}.
\begin{theorem}\label{thm:temperature}
Let $(\alpha_\pm,u_\pm,T_\pm)$ solve Rankine-Hugoniot conditions \eqref{eq:Rankine-Hugoniot-n}. If $\frac1D$ and $\frac{d}{D}$ are small, then
\begin{equation}\label{e:qtheta_-pprox}
    \frac{\theta_-}{\theta_+} \sim \frac{3}{2} + \frac{D}{3}\left(1 - \frac{2d}{D}\right).
\end{equation}
\end{theorem}
\begin{example}\label{nb:temperature}
For the values of Example \ref{ex:hydrogen_gas}, formula \eqref{e:qtheta_-pprox} shows $ \frac{\theta_-}{\theta_+} > 2.7068
 \times 10$, 
i.e., $T_{-} > 8.1205 {\rm K}.$
\end{example}

\section{Incident and reflected shock waves}\label{sec:condition Lax}

In this section, notation differs slightly from that introduced in Section \ref{sec:temperature}, see Figure \ref{fig:wall2}; this is what is labelled as Regime 5 in \cite{Fukuda-Okasaka-Fujimoto}.
The state behind the incident shock is that ahead of the reflected shock. We denote the speeds of the incident and reflected shock wave by $U_I$ and $-U_R$, respectively. The unperturbed states in front of the incident shock are denoted by $+$, those in the back by $-$, the states in the back of the reflected shock are denoted by $\sharp$. We clearly have $u_{\sharp} = 0$; since $u_+ = u_{\sharp} = 0,$ the reflector plays the same role of a  contact discontinuity.


\begin{figure}[htbp]
\begin{picture}(120,80)(0,-20)
\setlength{\unitlength}{1pt}

\put(120,0){
\put(30,-10){
\put(0,0){\vector(1,0){15}}
\put(15,-3){\makebox(0,0)[t]{$x$}}
\put(0,0){\vector(0,1){15}}
\put(-3,15){\makebox(0,0)[r]{$t$}}
}

\put(120,20){
\put(0,0){\line(0,1){40}}
\put(0,0){\line(0,-1){40}}
\put(0,0){\line(-1,1){40}}
\put(0,0){\line(-1,-1){40}}

\put(-10,-35){\makebox(0,0)[br]{$+$}}
\put(-30,0){\makebox(0,0)[r]{$-$}}
\put(-10,35){\makebox(0,0)[tr]{$\sharp$}}

\put(-40,-38){\makebox(0,0)[br]{$U_I$}}
\put(-39,38){\makebox(0,0)[tr]{$-U_R$}}
}

%
}
\end{picture}
\caption{\label{fig:wall2}{The shock reflection and the corresponding states in the $(x,t)$-plane.}}
\end{figure}
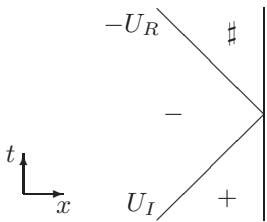
\par 
The main result of this section is that both incident and reflected shock waves satisfy the Lax shock condition under suitable conditions. In a genuine nonlinear region this is a consequence of \cite[Th. 18.2]{Smoller} because of \eqref{eq:euler characteristic speed} and \eqref{e:pS}.

\paragraph{The incident shock wave} 
In this case we have $u_->0$. As in \eqref{e:ThetaDeltaD}, we define 
\begin{equation}\label{e:D_+}
D_{+} = \frac{u_-^2}{a^2 \theta_+}.
\end{equation}

\begin{theorem}\label{thm:shock condition}
Assume $u_+=0$, $\alpha_+\sim0$ and $u_->0$. If $D_{+} > 4$ and
\begin{equation}\label{eq:shock condition+}
     \frac{\theta_-}{\theta_+} < (\sqrt{D_{+}} - 1)^2,
\end{equation} 
then the Lax shock conditions \eqref{eq:shock_conditions} are satisfied.
\end{theorem}
\begin{proof}
For brevity, below we simply denote $D$ for $D_{+}$. By \eqref{eq:shock_conditions} we must prove that 
\begin{equation}\label{eq:shock_conditions2}
     u_-< U_I < u_- + c_-, \quad  c_{+} < U_I.
\end{equation} 
\par
By \eqref{eq:EOS}, \eqref{eq:e-n} and $\eqref{e:RH-enthalpy}_2$ it follows $\rho_+\left(a^2\theta_+ + V_+^2\right) = \rho_-\left(a^2\theta_- + V_-^2\right)$. If $V_-=0$, then $V_+=0$ and everything is trivial. If $V_-\neq 0,$ then $\eqref{e:RH-enthalpy}_1$ yields
\begin{equation}\label{eq:thermodynamic RHV_pm}
   \frac{a^2\theta_+ + V_+^2}{a^2\theta_- + V_-^2} = \frac{\rho_-}{\rho_+} = \frac{V_+}{V_-}.
\end{equation}
We have $V_{+}= U_I$, $V_-= U_I - u_-$,
and by \eqref{eq:thermodynamic RHV_pm} we obtain
\begin{equation}\label{e:xyu}
 \frac{a^2\theta_- + (U_I - u_-)^2}{a^2\theta_+ + U_I^2}   = \frac{U_I - u_-}{U_I}.
\end{equation}
\par
Notice that by \eqref{e:xyu} it follows $U_I > u_-$ and then the first condition in \eqref{eq:shock_conditions2} is satisfied. 
This formula yields a quadratic equation for $U_I$, namely,
$$
  U_I^2 -\left[u_- + \frac{a^2}{u_- }(\theta_- - \theta_+)\right]U_I - a^2\theta_+= 0.
$$
Since $u_- >0$, we deduce that $U_I>0$ and then
\begin{eqnarray}\label{eq:U_I}
  U_I 
  &=&  \frac{u_-}{2}\left[{\textstyle1 + \frac{1}{D}\left(\frac{\theta_-}{\theta_+} - 1\right) + \sqrt{\left[ 1 + \frac{1}{D}\left(\frac{\theta_-}{\theta_+} - 1\right)\right]^{\! 2} + {\textstyle \frac{4}{D}}}}\right].\label{e:69}
\end{eqnarray}
Thus, the two last shock conditions in \eqref{eq:shock_conditions2} 
to be proved are (here we set $\alpha_+=0$)
\begin{equation}\label{eq:shock condition}
 a\sqrt{\frac{5\theta_+}{3}} <
  U_I  < u_- + a\sqrt{\frac{5\theta_-}{3}}\sqrt{ 1 - \frac{\frac{2\alpha_-(1 - \alpha_-)\Ti^2}{15T_-^2}}{1 + \alpha_-(1 - \alpha_-)\left(\frac{5}{4} + \frac{\Ti}{T_-} + \frac{\Ti^2}{3T_-^2}\right)}}.
\end{equation} 
\par
Since $\frac{\theta_-}{\theta_+} - 1>0$ by Lemma \ref{l:Theta>0}, by \eqref{e:69} we deduce
$$
 U_I > \frac{u_-}{2}\left(1 + \sqrt{1 + \frac{4}{D}}\right),
$$
and then a sufficient condition for the left inequality in \eqref{eq:shock condition} to hold is
$$
 a\sqrt{\frac{5\theta_+}{3}} <  \frac{u_-}{2}\left( 1 + \sqrt{1 + \frac{4}{D}}\right).
$$
By \eqref{e:D_+} this inequality is equivalent to 
$$
 \sqrt{\frac{5}{3}} <  \frac{\sqrt{D}}{2}\left( 1 + \sqrt{1 + \frac{4}{D}}\right)
\frac{\theta_-}{\theta_+}
$$
which is satisfied for $D > 0.266.$ Now, since 
$$
\sqrt{\left[1 + \frac{1}{D}\left(\frac{\theta_-}{\theta_+} - 1\right)\right]^2 + \frac{4}{D}} 
\leq 
\left[1 + \frac{1}{D}\left(\frac{\theta_-}{\theta_+} - 1\right)\right] + \frac{2}{\sqrt{D}},
$$
we deduce by \eqref{e:69} that 
$$
U_I < u_-\left[1 + \frac{1}{D}\left(\frac{\theta_-}{\theta_+} - 1\right) + \frac{1}{\sqrt{D}}\right].
$$
A sufficient condition for the inequality on the right in \eqref{eq:shock condition} to hold is
$$
   u_- \left[\frac{1}{D}\left(\frac{\theta_-}{\theta_+} - 1\right) + \frac{1}{\sqrt{D}}\right]
< a\sqrt{\frac{5\theta_-}{3}}\sqrt{\frac{1 + \alpha_-(1 - \alpha_-)\left(\frac{5}{4} + \frac{\Ti}{T_-} + \frac{\Ti^2}{5T_-^2}\right)}{1 + \alpha_-(1 - \alpha_-)\left(\frac{5}{4} + \frac{\Ti}{T_-} + \frac{\Ti^2}{3T_-^2}\right)}},
$$
that is,
\begin{eqnarray}
  \frac{1}{D}\left(\frac{\theta_-}{\theta_+} - 1\right) + \frac{1}{\sqrt{D}}
&<& 
\frac{1}{\sqrt{D}}\sqrt{\frac{5\theta_-}{3\theta_+}}\sqrt{\frac{1 + \alpha_-(1 - \alpha_-)\left(\frac{5}{4} + \frac{\Ti}{T_-} + \frac{\Ti^2}{5T_-^2}\right)}{1 + \alpha_-(1 - \alpha_-)\left(\frac{5}{4} + \frac{\Ti}{T_-} + \frac{\Ti^2}{3T_-^2}\right)}}.
\label{e:simple_condition}
\end{eqnarray}
Note that the argument of the last square root in \eqref{e:simple_condition}
is a decreasing function with respect to $\frac{\Ti}{T_-}$; hence,
$$
   \frac{3}{5} <  \frac{1 + \alpha_-(1 - \alpha_-)\left(\frac{5}{4} + \frac{\Ti}{T_-} + \frac{\Ti^2}{5T_-^2}\right)}{1 + \alpha_-(1 - \alpha_-)\left(\frac{5}{4} + \frac{\Ti}{T_-} + \frac{\Ti^2}{3T_-^2}\right)} \leq 1.
$$ 
Thus a simple sufficient condition in order that \eqref{e:simple_condition} holds is
$$
   \left[\frac{1}{D}\left(\frac{\theta_-}{\theta_+} - 1\right) + \frac{1}{\sqrt{D}}\right]
< \frac{1}{\sqrt{D}}\sqrt{\frac{\theta_-}{\theta_+}}.
$$
We denote $X = \sqrt{\theta_-/\theta_+}$; then $X\ge1$ by Lemma \ref{l:Theta>0} and the last inequality above is equivalent to 
$
  X^2 - \sqrt{D}X + \sqrt{D} -1 < 0.
$
The solutions of this inequality are $1 < X <  \sqrt{D} -1$ and they exists because $D>4$. This proves the theorem.
\end{proof}
\begin{nb}\label{nb:shock condition}
In the case $u_- = 1.6 \times 10^4${\rm m s}${}^{-1}$ and $T_{+} = 300 ${\rm K}, estimate \eqref{e:thetaless} gives $\frac{\theta_-}{\theta_+}<3.5708 \times 10$. We find $(\sqrt{D_+} - 1)^2 = 8.337 \times 10,$
which shows that the shock conditions are satisfied.
\end{nb}

\paragraph{The reflected shock wave} 
We shall apply Theorem \ref{thm:Hugoniot curve alpha T}  to the reflected wave by setting $(\alpha_0, u_0, T_0) = (\alpha_-, u_-, T_-)$ and $(\alpha, u, T)  = (\alpha_{\sharp}, u_{\sharp},T_{\sharp}).$ Then for given $\alpha_-, u_-, T_-$ (obtained in the previous section) and $u_{\sharp} = 0,$ we find that there is a unique set of the solution $(\alpha_{\sharp}, T_{\sharp}).$ We denote
\begin{equation}\label{e:DA}
D_- = \frac{u_-^2}{a^2 \theta_-}.
\end{equation}

\begin{figure}[hbt]
\centering
 \includegraphics[width =.60\linewidth]{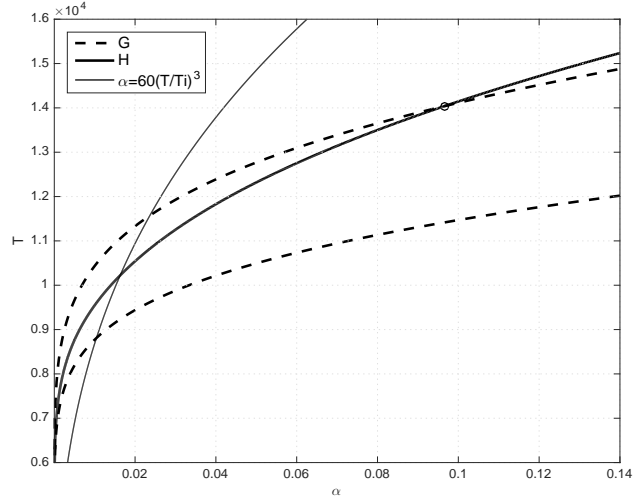}
 \caption{The reflected shock wave: the Hugoniot locus $H_+(\alpha,T)=0$ and $G_+(\alpha, T) = 0$. Here 
 $\alpha_- = 0.0101$, $u_+=0 $, $T_{-} = 9559.53$ and $u_- = 1.6 \times 10^4${\rm m s}${}^{-1}$. The intersection point is
 $\alpha_\sharp = 0.0965$, $T_\sharp = 14042.0$, where genuine nonlinearity holds.}
 \label{fig:Hugoniot+G(a,T)_16000_reflection}
\end{figure}

\begin{theorem}\label{thm:shock condition_reflection}
If $D_-\ge\frac13$, then the Lax shock conditions hold for the reflected shock wave.
\end{theorem}

\begin{proof}
As in the proof of Theorem \ref{thm:shock condition}, we simply denote $D$ for $D_-$. By \eqref{eq:shock_conditions1} we must prove that 
\begin{equation}\label{eq:shock_conditions-UR}
     c_--u_-< U_R < c_\sharp, \quad  U_R>0.
\end{equation}
We have $	V_- =- U_R - u_-$, $V_\sharp = - U_R$; 
%
by \eqref{eq:thermodynamic RHV_pm} we deduce
\begin{equation}\label{e:UR>0}
 \frac{a^2\theta_\sharp + U_R^2}{a^2\theta_- + (U_R + u_-)^2}   = \frac{U_R}{U_R + u_-}
\end{equation}
and thus we obtain a quadratic equation for $U_R$:
$$
  U_R^2 -\left[- u_- + \frac{a^2}{u_- }(\theta_\sharp - \theta_-)\right]U_R - a^2\theta_\sharp= 0.
$$
Since $U_R > 0$ by \eqref{e:UR>0}, the third condition in \eqref{eq:shock_conditions-UR} is satisfied.  By \eqref{e:DA} we obtain
\begin{eqnarray}
U_R &=&
\frac{u_-}{2}\left[{\textstyle - 1 + \frac{1}{D}\left(\frac{\theta_\sharp}{\theta_-} - 1\right) + \sqrt{\left[- 1 + \frac{1}{D}\left(\frac{\theta_\sharp}{\theta_-} - 1\right)\right]^{\! 2} + {\textstyle \frac{4}{D}\frac{\theta_\sharp}{\theta_-}}}}\right].
\label{eq:U_R}
\end{eqnarray}
By \eqref{eq:euler characteristic speed}, the shock condition in \eqref{eq:shock_conditions-UR} to be proved is
\begin{equation}\label{eq:shock condition_R}
-u_- + a\sqrt{\frac{5\theta_-}{3}}\sqrt{\frac{1 + \alpha_-(1 - \alpha_-)\left(\frac{5}{4} + \frac{\Ti}{T_-} + \frac{\Ti^2}{5T_-^2}\right)}{1 + \alpha_-(1 - \alpha_-)\left(\frac{5}{4} + \frac{\Ti}{T_-} + \frac{\Ti^2}{3T_-^2}\right)}}  
<
U_R 
<
a\sqrt{\frac{5\theta_\sharp}{3}}\sqrt{\frac{1 + \alpha_\sharp(1 - \alpha_\sharp)\left(\frac{5}{4} + \frac{\Ti}{T_\sharp} + \frac{\Ti^2}{5T_\sharp^2}\right)}{1 + \alpha_\sharp(1 - \alpha_\sharp)\left(\frac{5}{4} + \frac{\Ti}{T_\sharp} + \frac{\Ti^2}{3T_\sharp^2}\right)}}.
\end{equation}
Since we are assuming $D_-\ge\frac13$ then by Corollary \ref{c:ThetaD} and \eqref{eq:U_R} we have
$$
  U_R \leq \frac{u_-}{2}\left\{{\textstyle - 1 + \frac{1}{D}\left(\frac{\theta_\sharp}{\theta_-} - 1\right) + \left[1 - \frac{1}{D}\left(\frac{\theta_\sharp}{\theta_-} - 1\right)\right] + {\textstyle \frac{2}{\sqrt{D}}\sqrt{\frac{\theta_\sharp}{\theta_-}}}}\right\} = \frac{u_-}{\sqrt{D}}\sqrt{\frac{\theta_\sharp}{\theta_-}}.
$$
Then, a sufficient condition for the left inequality in \eqref{eq:shock condition_R} is
$$
    \frac{u_-}{\sqrt{D}}\sqrt{\frac{\theta_\sharp}{\theta_-}}  < a\sqrt{\frac{5\theta_\sharp}{3}}\sqrt{\frac{1 + \alpha_\sharp(1 - \alpha_\sharp)\left(\frac{5}{4} + \frac{\Ti}{T_\sharp} + \frac{\Ti^2}{5T_\sharp^2}\right)}{1 + \alpha_\sharp(1 - \alpha_\sharp)\left(\frac{5}{4} + \frac{\Ti}{T_\sharp} + \frac{\Ti^2}{3T_\sharp^2}\right)}}
$$
and then, by recalling \eqref{e:DA}, a simple sufficient condition for the left inequality in \eqref{eq:shock condition_R} is
$$
\frac{3}{5}\left[1 + \alpha_\sharp(1 - \alpha_\sharp)\left(\frac{5}{4} + \frac{\Ti}{T_\sharp} + \frac{\Ti^2}{3T_\sharp^2}\right)\right] <  1 + \alpha_\sharp(1 - \alpha_\sharp)\left(\frac{5}{4} + \frac{\Ti}{T_\sharp} + \frac{\Ti^2}{5T_\sharp^2}\right),
$$
which is obviously true. About the inequality on the right in \eqref{eq:shock condition_R}, 
 notice that by \eqref{eq:Theta2} we have $\frac{\theta_\sharp}{\theta_-} \geq 1$; by \eqref{eq:U_R} we deduce
\begin{eqnarray*}
 U_R + u_- &=& \frac{1}{2}\left[{\textstyle u_- + \frac{a^2\theta_-}{u_-}\left(\frac{\theta_\sharp}{\theta_-} - 1\right) + u_-\sqrt{\left[- 1 + \frac{1}{D}\left(\frac{\theta_\sharp}{\theta_-} - 1\right)\right]^{\! 2} + {\textstyle \frac{4}{D}\frac{\theta_\sharp}{\theta_-}}}}\right]\\
 &\geq& {\textstyle a \sqrt{\theta_-}\sqrt{\frac{\theta_\sharp}{\theta_-} - 1}  + \frac{u_-}{2}\sqrt{\left[- 1 + \frac{1}{D}\left(\frac{\theta_\sharp}{\theta_-} - 1\right)\right]^{\! 2} + {\textstyle \frac{4}{D}\frac{\theta_\sharp}{\theta_-}}}}
\end{eqnarray*}
and a simple sufficient condition for condition \eqref{eq:shock condition_R} to hold is
$$
   {\textstyle a \sqrt{\theta_-}\sqrt{\frac{\theta_\sharp}{\theta_-} - 1}  + \frac{u_-}{2}\sqrt{\left[- 1 + \frac{1}{D}\left(\frac{\theta_\sharp}{\theta_-} - 1\right)\right]^{\! 2} + {\textstyle \frac{4}{D}\frac{\theta_\sharp}{\theta_-}}}} 
> a\sqrt{\frac{5\theta_-}{3}},
$$
that is,
$$
    {\textstyle \sqrt{\frac{\theta_\sharp}{\theta_-} - 1}  + \sqrt{D\left[- 1 + \frac{1}{D}\left(\frac{\theta_\sharp}{\theta_-} - 1\right)\right]^{\! 2} + {\textstyle \frac{4\theta_\sharp}{\theta_-}}}} 
> \sqrt{\frac{5}{3}}.
$$
The left side of the above inequality is larger than 2. Thus also the inequality on the right is satisfied.
\end{proof}

\begin{example}
In the case $\alpha_- = 0.0101$, $T_{-} = 9559.53$ and $u_- = 1.6 \times 10^4${\rm m s}${}^{-1}$, we compute $D_-= 3.1862$. So the assumption of Theorem \ref{thm:shock condition_reflection} is largely satisfied.
\end{example}
\begin{nb}
The degree of ionization behind the reflected shock wave is  $\alpha_\sharp = 0.0965$, which is much higher than $\alpha_- = 0.0101;$ on the other hand, we have that $T_\sharp<\frac32 T_-.$ These different growths are due to the concavity of the Hugoniot locus, see Corollary \ref{cor:concave T}. As the above example indicates, the Hugoniot locus of $(\alpha_-, T_-)$ is a low-gradient curve compared with that of the incident shock wave. Thus we conclude that the reflector set at the end of the stock tube has a notable effect on the increasing of the degree of ionization.
\end{nb}
\section{Conclusions and discussions}\label{sec:conclusions}
In this paper we examined some mathematical properties of both incident and reflected shock waves in an electromagnetic shock tube with a reflector, where the ionization of the gas is taken into account. 
\par
In an actual shock tube problem, the degree of ionization $\alpha_+$ of the initial state is almost zero. By assuming that $\alpha_+$ and  $\exp(-\Ti/T_+)$ are comparable, which is confirmed by experimental data, we obtained an approximate form of both the Rankine-Hugoniot conditions and the Hugoniot locus issuing from $(0, T_-)$ in the $(\alpha, T)$-plane. We showed that, under a reasonable limitation that is in accord with experimental data, the approximate Hugoniot locus is located in a genuinely nonlinear region as is the case for both shock waves.
\par
We identified the dimensionless parameters $D=\frac{ (u_- - u_+)^2}{a^2\theta_+}$ and $d=\frac{\Ti}{\theta_+}(\alpha_- - \alpha_+)$; in particular $D$ is computed only by the initial quantities. Then we computed and estimated the ratio $\frac{T_-(1 + \alpha_-)}{T_+(1 + \alpha_+)}$ for the incident shock wave. We showed that if the degree of ionization is sufficiently small through the whole process and $d$ is small, then the ratio $\frac{T_-}{T_+}$ is estimated only by $D.$ A computation based on the data used in \cite{Fukuda-Okasaka-Fujimoto} indicates that, in the state behind the incident shock front, the temperature increases remarkably but the degree of ionization does not undergo a similar growth.
\par
The reflected shock wave is constructed by exploiting an existence theorem proved in \cite{Asakura-Corli_ionized}. Under some reasonable limitations, which fully agree with experimental data, we prove that both shock waves satisfy the Lax conditions. We emphasize that this is valid even outside genuinely nonlinear regions.
\par
Moreover, we show that in the state behind the reflected shock front, the degree of ionization increases remarkably. This phenomenon and the analogous previous one are mathematically translated by showing that the Hugoniot loci are locally concave. It is proved that the Hugoniot locus issuing from a very small $\alpha_0$ and $T_0$ is a very high-gradient curve for small $\alpha$; on the contrary, the curve tends to a low-gradient curve as $\alpha$ becomes bigger. This proves that the reflector strongly increases the degree of ionization.

\appendix
\section{Some numerical values}\label{sec:appendix_numerical_values}
\small We collect some physical values; they are referred to the case of a hydrogen gas.  Specific gas constant:  $a^2=8314 {\rm\, J\, kg}{}^{-1} {\rm K}{}^{-1}$; constant in Saha's equation $\kappa = 29.9774 {\rm\, K^\frac52\ m\ kg}{}^{-1}\ {\rm s}{}^{-2}$; ionization temperature $\Ti = 1.5780 \times 10^5{\rm K}$. The following values are used in the low-pressure chamber in \cite{Fukuda-Okasaka-Fujimoto}, for $\alpha$ computed through \eqref{eq:deg-ionization-n}: $u_+=0$, $p=1466.3 {\rm Pa}\ (11 {\rm Torr})$, $u_- = 8.1 \times 10^3$m s${}^{-1}$. We also consider the speed $u_- = 1.6 \times 10^4$m s${}^{-1}$. Other reference values are as follows:
%
\begin{center}
\begin{tabular}{c | c c c c c c c}
& $T$& $\alpha$ & $\frac{T}{T_{\rm i}}$ & $60(\frac{T}{T_{\rm i}})^3$ & $e^{-\frac{T_{\rm i}}{2T}}$ & $(\frac{T}{T_{\rm i}})^\frac54 e^{-\frac{T_{\rm i}}{2T}}$ &$\sqrt{\frac{T_{+}T_{\rm i}^{\frac{3}{2}}}{\kappa p_{+} }}$ 
\\[2mm]
\hline
(1) & $300{\rm K}$ & $3.5929 \times 10^{-114}$ & $0.0019$ &$4.1228\times 10^{-7}$ & $6.0332 \times 10^{-115}$ & $2.3950 \times 10^{-118}$&$6.5408\times 10^2$
\\[2mm]
(2) & $750{\rm K}$ & $3.8418 \times 10^{-45}$ & $0.0047$ &$6.4419\times 10^{-6}$ & $2.0522 \times 10^{-46}$ & $2.5610 \times 10^{-49}$&$1.0342\times 10^3$
\\
\hline
\end{tabular}
\end{center}


\section*{Acknowledgements}
The second author is member of the Italian GNAMPA-INDAM and acknowledges support from this institution. He was also supported by the PRIN Project {\em Nonlinear partial differential equations of hyperbolic or dispersive type and transport equations.}



\end{document}